\documentclass[graybox]{svmult}

\usepackage{mathptmx}       
\usepackage{helvet}         
\usepackage{courier}        
\usepackage{type1cm}        
\usepackage{makeidx}         
\usepackage{graphicx}        
\usepackage{multicol}        
\usepackage[bottom]{footmisc}
\usepackage{amsmath, amsfonts, stmaryrd, amssymb, dsfont}

\newcommand{\N}{\mathbb{N}}
\newcommand{\Z}{\mathbb{Z}}

\newcommand{\R}{\mathbb{R}}

\renewcommand{\P}{\mathbb{P}}
\renewcommand{\E}{\mathbb{E}}

\newcommand{\mc}[1]{{\mathcal #1}}
\newcommand{\mf}[1]{{\mathfrak #1}}

\newcommand{\bb}[1]{{\mathbb #1}}

\renewcommand{\epsilon}{\varepsilon}

\usepackage{color}

\newcommand{\diff}{\,\mathrm{d}}

\makeindex             

\begin{document}

\title*{Coupling hydrodynamics of several Facilitated
Exclusion Processes with closed boundaries}
\titlerunning{
Hydrodynamics of
FEP
with closed boundaries}

\author{Hugo Da Cunha and Lu Xu}
\institute{Hugo Da Cunha \at Université Claude Bernard Lyon 1, CNRS, Ecole Centrale de Lyon, INSA Lyon, Université Jean Monnet, ICJ UMR5208,
69622 Villeurbanne, France. \at \email{dacunha@math.univ-lyon1.fr}
\and 
Lu Xu \at Gran Sasso Science Institute, Viale Luigi Rendina 26, 67100 L'Aquila, Italy \at \email{lu.xu@gssi.it}}

\maketitle

\abstract{In this paper, we prove the hydrodynamic limit for the ergodic dynamics of Facilitated Exclusion Process with closed boundaries in the symmetric, asymmetric and weakly asymmetric regimes. For this, we couple it with a Simple Exclusion Process by constructing a mapping that transforms the facilitated dynamics into the simple one, inspired by \cite{Goldstein_Lebowitz}. As the hydrodynamic behaviour of simple exclusion process has been extensively studied,
we can deduce the corresponding hydrodynamics for the facilitated exclusion process.}


\section{Introduction}
\label{sec:introduction}

The Facilitated Exclusion Process is a model of \emph{kinetically constrained stochastic lattice gas} introduced in the physics literature in \cite{RossiPastor} to model some mechanisms of solid/liquid interfaces. This model exhibits a phase transition at the critical particle density $\frac12$, separating an absorbing subcritical phase from an active supercritical phase. In recent years, it has drawn the attention of the mathematical community mostly on the one-dimensional lattice $\Z$ or the torus. For the symmetric dynamics on the periodic torus, the macroscopic behaviour, more commonly known as \emph{hydrodynamic limit}, is given by a fast diffusion equation in the high-density phase \cite{BESS20}, or by a Stefan problem separating both phases \cite{StefanProblem}. The hydrodynamic limit for the asymmetric dynamics has also been investigated on the full line in \cite{mappingHydrodynamics} by using a mapping towards a zero-range process which, on the contrary to FEP, has the nice property of being attractive. The fluctuations around equilibrium of all of these models, together with the weakly asymmetric dynamics, have been examined in \cite{ErignouxZhao}. Transience times for the FEP in finite volume have also been extensively studied in \cite{Chleboun,Massoulie}. It is only recently that the FEP has been considered in presence of boundary conditions. In \cite{DCES}, the authors consider the symmetric FEP in contact with boundary reservoirs that can inject/remove particles in the system, and prove that depending on the strength of the boundary interactions, different boundary conditions are imposed on the fast diffusion equation ruling the hydrodynamic limit in the high-density phase. 

In this paper, we investigate symmetric, asymmetric and weakly asymmetric FEPs on the one-dimensional finite lattice with \emph{closed boundaries}, \textit{i.e.}~when impermeable walls are placed at both boundaries, not allowing particles to enter or escape the bulk. 
Inspired by the approach in \cite{mappingHydrodynamics}, we study the hydrodynamic limit via a one-to-one correspondence, similar to the one constructed in \cite{Goldstein_Lebowitz}, between FEP in its ergodic component and Simple Exclusion Process (SEP). As the hydrodynamic behaviour of SEP has been well studied in a variety of contexts \cite{BMNS17,CapitaoGoncalves,NotesIHP,Xu_hydrodynamics}, the one of FEP can be deduced correspondingly.

\medskip

To be more precise about the model and result, let $\sigma$, $p$ and $\kappa$ be non-negative parameters, with $\sigma\neq 0$. In the model considered here, each particle performs a random walk on the finite lattice $\{1,\hdots ,N-1\}$ satisfying that there can be at most one particle per site (\emph{exclusion rule}), and the following \emph{facilitated kinetic constraint}:
\begin{itemize}
    \item a particle lying at a site $x$ can jump to site $x+1$ provided there is a particle at site $x-1$, and in this case the jump occurs at rate $\sigma +pN^{-\kappa}$;
    \item a particle lying at site $x$ can jump to site $x-1$ provided there is a particle at site $x+1$, and in this case the jump occurs at rate $\sigma$.
\end{itemize}
At both ends of the system there are impermeable walls, meaning that a particle at site $1$ (resp.~$N-1$) cannot jump to the left (resp.~right). However, we assume that it can always jump to the right (resp.~left) with rate $\sigma +pN^{-\kappa}$ (resp.~$\sigma$). Then, we accelerate this process with a time-scale $\Theta_N=N^{(1+\kappa )\wedge 2}$. We prove the hydrodynamic limit of this model when it starts straight from its high-density phase, and even when it has already reached its ergodic component to not care about the transience time (partly estimated in \cite{Chleboun} for the segment). Namely, we show that under suitable conditions on the initial state, the macroscopic density of particles $\rho=\rho_t(u)$ in the system evolves according to the PDE
\begin{equation*}
    \partial_t\rho + p\mathds{1}_{\{\kappa\le 1\}}\partial_u\left(\frac{(1-\rho )(2\rho -1)}{\rho}\right) =\sigma\mathds{1}_{\{\kappa\ge 1\}}\partial_u^2\left(\frac{2\rho -1}{\rho}\right)
\end{equation*}
for $u \in [0,1]$, with boundary conditions \emph{formally} given by
\begin{equation*}
    \left[ p\mathds{1}_{\{\kappa\le 1\}}\frac{(1-\rho )(2\rho -1)}{\rho} - \sigma\mathds{1}_{\{\kappa\ge 1\}}\partial_u\left(\frac{2\rho -1}{\rho}\right) \right] \bigg|_{u=0,1} = 0.
\end{equation*}
We underline that when the PDE is parabolic ($\kappa\ge1$), the attached boundary conditions are of Neumann-type if $\kappa>1$ and Robin-type if $\kappa=1$.
However, if the PDE is hyperbolic ($\kappa<1$), the boundary conditions shall be understood in the sense of \cite{Otto}, see Definition \ref{defin:entropysol} in Section \ref{appendix} for details.

\medskip

This article is organised as follows. In Section \ref{sec:models} we introduce the process that we study throughout the paper, namely the Facilitated Exclusion Process with closed boundaries. In Section \ref{sec:results} we build a microscopic mapping between the dynamics of the FEP and the dynamics of the SEP, and we state the main result of the paper, Theorem \ref{thm:HLFEPs}, that is the hydrodynamic for FEP with closed boundaries. In Section \ref{sec:proof}, we explain how this macroscopic mapping between dynamics induces a macroscopic mapping between particles densities, and how this allows to deduce Theorem \ref{thm:HLFEPs} from the already known hydrodynamic limits of SEP. Section \ref{sec:mappingsolutions} is dedicated to prove that, through the macroscopic mapping between densities, the solutions of the hydrodynamic equations of the SEP are sent to the solutions of the hydrodynamic equations of the FEP. Finally, Section \ref{appendix} serves as an appendix listing all the notions of solutions used for the previously introduced hydrodynamic equations, and giving some useful results about entropy solutions.

\section{Model and notations}
\label{sec:models}

Let $N\in\N$ be a scaling parameter that shall go to infinity. Define the one-dimensional lattice $\Lambda_N=\lbrace 1,\hdots ,N-1\rbrace$ called \emph{bulk}, and also $\Omega_N=\{0,1\}^{\Lambda_N}$ the space of configurations. A configuration of particles is an element $\eta=(\eta_x)_{x\in\Lambda_N}\in\Omega_N$, where the occupation variable $\eta_x\in\{0,1÷\}$ indicates whether there is a particle at site $x$ or not. If $\eta\in \Omega_N$ is a configuration, and $x,x'\in\Lambda_N$ are two sites, then we define the configuration $\eta^{x,x'}$ to be the one obtained from $\eta$ by exchanging the two occupation variables $\eta_x$ and $\eta_{x'}$.



\subsection{Facilitated Exclusion Processes}
\label{sec:FEPs}

In this article, we focus on the \emph{Facilitated Exclusion Process} (FEP) on $\Lambda_N$.
It encodes different settings (symmetric, asymmetric and weakly asymmetric) but always with closed (impermeable) boundaries.

Define the \emph{symmetric} and \emph{totally asymmetric} facilitated exclusion processes to be the Markov processes on $\Omega_N$ driven respectively by the infinitesimal generators $\mathcal{L}_N^\mathrm{S}$ and $\mathcal{L}_N^\mathrm{TA}$ acting on functions $f:\Omega_N\longrightarrow\R$ through the formulas
\begin{align}
    &\mathcal{L}_N^\mathrm{S}f(\eta ) = \sum_{x=1}^{N-2} \big( \eta_{x-1}\eta_x(1-\eta_{x+1})+(1-\eta_x)\eta_{x+1}\eta_{x+2}\big)\big[ f(\eta^{x,x+1})-f(\eta )\big],\\
    &\mathcal{L}_N^\mathrm{TA}f(\eta )= \sum_{x=1}^{N-2} \eta_{x-1}\eta_x(1-\eta_{x+1})\big[ f(\eta^{x,x+1})-f(\eta )\big].
\end{align}
In these sums, we set $\eta_0=\eta_N=1$ by convention, in order to ensure that a particle that is at the border of the system is always able to jump towards the bulk without undergoing any constraint. In fact, the impermeable walls at the boundaries are seen as particles that are there and cannot move.

Let $\sigma >0$ and $p,\kappa \ge 0$, and set $\Theta_N=N^{(1+\kappa)\land2}$. Consider the Markov process driven by the infinitesimal generator
\begin{equation}\label{def:genFEP}
    \mathcal{L}_N := \Theta_N \big( \sigma\mathcal{L}_N^\mathrm{S}+pN^{-\kappa}\mathcal{L}_N^\mathrm{TA}\big).
\end{equation}
This process has the advantage of encompassing several models, namely in the \emph{diffusive case} $\Theta_N=N^2$:
\begin{itemize}
    \item \textbf{(SFEP)} the \emph{symmetric facilitated exclusion process} when $p=0$ and $\sigma >0$;
    \item \textbf{(vWAFEP)} the \emph{very weakly asymmetric facilitated exclusion process} when $p>0$, $\sigma >0$ and $\kappa >1$;
    \item \textbf{(WAFEP)} the \emph{weakly asymmetric facilitated exclusion process} when $p >0$, $\sigma>0$ and $\kappa =1$;
\end{itemize}
and in the \emph{hyperbolic case} $\Theta_N=N^{1+\kappa}$:
\begin{itemize}
    \item \textbf{(AFEPvv)} the \emph{asymmetric facilitated exclusion process with vanishing viscosity} when $p >0$, $\sigma>0$ and $0\le\kappa<1$.
\end{itemize}
We investigate the hydrodynamic limit of the corresponding Markov process for all of the cases listed above.
Only for \textbf{(AFEPvv)} we require an extra technical condition $\kappa>\frac12$, see Theorem \ref{thm:HLFEPs} and Proposition \ref{prop:HLSEPs}.

\subsection{Ergodicity property}
\label{sec:ergodicityFEP}

Because of the kinetic constraint, FEP displays two distinct phases
: at density greater than $\frac12$, after a transience time, the process reaches an \emph{ergodic component} that we describe hereafter; whereas at density smaller than or equal to $\frac12$, after a transience time, the process freezes because all particles become isolated. For more details about the phase transition, see \cite[Section 2.2.1]{BESS20}.

In this paper, we focus on the FEPs
when they have already reached their ergodic component.
The ergodic component $\mc{E}_N$ is a subset of $\Omega_N$ that is made of the configurations in which all empty sites are isolated: 
\begin{equation}
    \label{def:EN}
    \mc{E}_N:=\big\{ \eta\in\Omega_N\; :\; \forall \{ x,x+1\}\subset\Lambda_N,\; \eta_x+\eta_{x+1}\ge 1\big\} .
\end{equation}
Given $k \le N$, let $\Omega_N^k$ be the hyperplane of configurations containing $k$ particles, and $\mc{E}_N^k$ be the set of ergodic ones in $\Omega_N^k$:
\begin{equation}
    \label{def:OmegaNk_ENk}
    \Omega_N^k:= \Big\{\eta\in\Omega_N\; :\; |\eta|:=\sum_{x\in\Lambda_N}\eta_x =k \Big\},\qquad \mc{E}_N^k:=\mc{E}_N\cap\Omega_N^k.
\end{equation}
Notice that the set $\mathcal{E}_N^k$ is non-empty if and only if $k\ge \big\lfloor \frac{N-1}{2}\big\rfloor$.

%

\section{Main results}
\label{sec:results}

\subsection{Coupling with Simple Exclusion Process}
\label{sec:micromap}

Coupling between FEP and other more tractable particle systems, \textit{e.g.}~simple exclusion on infinite lattice \cite{Goldstein_Lebowitz} and zero-range process \cite{mappingHydrodynamics}, has been used as an efficient method to study the statistical properties of FEP.
We generalize the coupling constructed in \cite{Goldstein_Lebowitz} on the full line to the case of finite lattice $\Lambda_N$.
The construction relies heavily on the fact that the total number of particles in the system is conserved along the dynamics:
\begin{equation*}
    \sum_{x\in\Lambda_N}\eta_x(t) = \sum_{x\in\Lambda_N}\eta_x(0), \qquad \forall\,t\ge 0.
\end{equation*}
This is the reason why we adopt the impermeable (closed) boundaries.
Another crucial feature that is necessary is that $\eta(0)$ lies in the ergodic component $\mathcal{E}_N$, and hence so does $\eta(t)$ for all $t>0$

Let $k\ge \big\lfloor \frac{N-1}{2}\big\rfloor$. We construct a function $\phi$ that maps each configuration $\eta\in\mc{E}_N^k$ to another configuration $\xi\in\Omega_M^\ell$ for some $(M,\ell)$ depending on $(N,k)$, as follows:
\begin{enumerate}
    \item First, we extend $\eta$ to a configuration $\bar\eta$ on $\{ 0,\dots N\}$ by adding the two fictive particles relative to the convention $\bar\eta_0=\bar\eta_N=1$ at both ends.
    \item Then, we number all the particles in $\{0,\dots ,N-1\}$ in $\bar\eta$ from $1$ to $M-1$.
    \item Finally, we define the configuration $\xi\in\Omega_M$ saying that the site $y$ in $\xi$ is occupied by a particle if, and only if, the $y$-th particle in $\bar\eta$ has a particle on its right neighbouring site.
\end{enumerate}
With this definition, we should have $M-1=k+1$ and the obtained configuration $\xi=\phi(\eta)$ has a number of particles equal to $k+1-(N-1-k)=2k-N+2$.
Hence, $(M,\ell)=(k+2,2k-N+2)$.
One can easily check that $\phi$ defines a one-to-one mapping between $\mc{E}_N^k$ and $\Omega_{k+2}^{2k-N+2}$, and that its inverse $\phi^{-1}$ can be described as follows:
\begin{enumerate}
    \item Take a configuration $\xi\in\Omega_{k+2}^{2k-N+2}$, and replace each of its `$0$' by `$10$'.
    \item Define $\eta=\phi^{-1}(\xi)$ to be the configuration obtained from it after removing the leftmost site (which is necessarily an occupied site `$1$').
\end{enumerate}
Since $\phi$ is well-defined on all non-empty $\mc{E}_N^k$, it naturally extends to an invertible map
\begin{equation}\label{eq:state-space}
\displaystyle
\phi: \mc{E}_N \longrightarrow \phi\big(\mc{E}_N\big) := \bigcup_{M=\lfloor\frac{N+1}{2}\rfloor}^{N+1} \Omega_M^{2M-N-2}.
\end{equation}

\noindent We plotted below an example of $\phi$ for a configuration $\eta\in\mathcal{E}_{15}^9$.
\begin{figure}[hbtp]
    \centering
    \includegraphics[width=\textwidth]{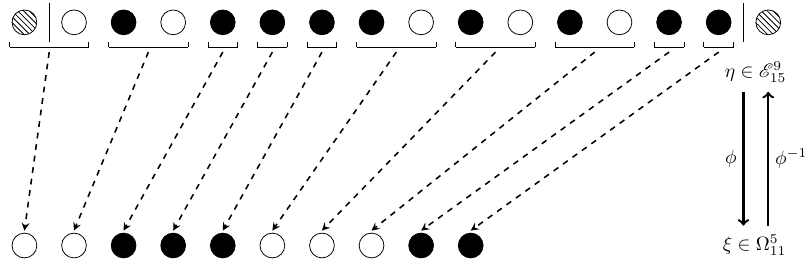}
    \label{fig:mapping}
    \caption{Illustration of the mapping $\phi$ on a configuration $\eta\in\mc{E}_{15}^9$. The hatched particles represent the fictive ones added at both ends of the configuration.}
\end{figure}

The map $\phi$ induces a map that gives a correspondence between the particles in $\eta$ and the sites in $\xi=\phi(\eta)$.
To see this, for $\eta\in\mc{E}_N$ let
\begin{equation}
    \label{def:Part(eta)}
    \mathrm{Part}(\eta)=\{x\in\Lambda_N\,:\,\eta_x=1\}
\end{equation}
be the set of coordinates of particles in $\eta$.
Notice that in $\xi=\phi (\eta )$, the first site always comes from the fictive particle which is added at site $0$, so one can define a bijection $\gamma_\xi: \Lambda_M\setminus\{1\}\longrightarrow\mathrm{Part}(\eta)$ as follows:
\begin{equation}
    \label{def:gammaxi}
    \gamma_\xi (y) = \sum_{z=1}^{y-1}\big( \xi_{z}+2(1-\xi_{z})\big) = \sum_{z=1}^{y-1}(2-\xi_{z}), \quad \forall\,y=2,\ldots,M-1.
\end{equation}
because the position $\gamma_\xi (y)$ of the particle that gave site $y\ge 2$ is equal to the number of particles, plus twice the number of holes before $y$ in $\xi$.

The mapping $\phi$ transforms the dynamics of the FEP into a dynamics without kinetic constraint, namely the Simple Exclusion Process (SEP).

\begin{theorem}[Microscopic mapping]
    \label{thm:micromap}
    Let $(\eta (t))_{t\ge 0}$ be the Markov process driven by the generator $\mc{L}_N$ defined in \eqref{def:genFEP}, and starting from an ergodic configuration $\eta (0)\in\mc{E}_N^k$.
    Then, $\xi(t):=\phi(\eta(t))$ is a Markov process evolving on $\Omega_M^\ell$ with $(M,\ell)=(k+2,2k-N+2)$, and its generator reads
    \begin{equation}
        \label{def:genSEP}
        \bb{L}_N=\Theta_N\big(\sigma\bb{L}_M^\mathrm{S}+pN^{-\kappa}\bb{L}_M^\mathrm{TA}\big),
    \end{equation}
    where for $f:\Omega_M^\ell\longrightarrow\mathbb R$, $\bb{L}_M^\mathrm{S}$ and $\bb{L}_M^\mathrm{TA}$ are respectively given by
    \begin{align}
    &\bb{L}_M^\mathrm{S}f(\xi) = \sum_{y=1}^{M-2} \big[ f(\xi^{y,y+1})-f(\xi )\big], \label{def:LSEP}\\
    &\bb{L}_M^\mathrm{TA}f(\xi )= \sum_{y=1}^{N-2} \xi_y(1-\xi_{y+1})\big[ f(\xi^{y,y+1})-f(\xi )\big]. \label{def:LTASEP}
    \end{align}
    Namely, $(\xi(t))_{t\ge0}$ is a simple exclusion process on $\Lambda_M$, in which a particle jumps to its right (resp.~left) neighbour with rate $\sigma+pN^{-\kappa}$ (resp.~$\sigma$).
\end{theorem}

Theorem \ref{thm:micromap} is proved by straightforwardly checking all possible transitions, see also \cite{Goldstein_Lebowitz}.
We omit the details here.

\begin{remark}
Similarly to the classification of FEP in Section \ref{sec:FEPs}, we call respectively \textbf{(SSEP)}, \textbf{(vWASEP)}, \textbf{(WASEP)} and \textbf{(ASEPvv)} the processes generated by $\bb{L}_N$ in different regimes of $\sigma$, $p$ and $\kappa$.
\end{remark}


\subsection{Hydrodynamic limit}
\label{sec:hydrodynamic}

With the map $\phi$ defined in the previous section and the knowledge on the macroscopic behaviours of SEP, we can investigate the hydrodynamic limit for FEP in different regimes.

Let us begin with the PDEs that will be obtained as the limit evolution equation of FEP.
To this end, define two functions $\mf{a}$ and $\mf{h}$ on $\big[\frac12,1\big]$ by
\begin{equation}
    \label{def:functions}
    \mf{a}(r):=\frac{2r-1}{r}, \qquad \mf{h}(r):=\frac{(1-r)(2r-1)}{r}.
\end{equation}
Observe that they are smooth, non-negative and $\mf{h}$ is concave.
Let $\rho^\mathrm{ini}: [0,1]\longrightarrow \big(\frac12 ,1\big]$ be a continuous function.
We list below the initial--boundary problems of three different PDEs.
\begin{itemize}
        \item The fast diffusion equation with Neumann boundary conditions
        \begin{equation}
            \label{FastDiffEqNeumann}
            \left\{
            \begin{aligned}
               &\,\partial_t\rho_t = \sigma\partial_u^2\mf{a}(\rho_t), \quad
               \rho_0 = \rho^\mathrm{ini},\\
               &\,\partial_u\mf{a}(\rho_t)(0)=\partial_u\mf{a}(\rho_t)(1)=0, \quad \forall\,t\in [0,T].
            \end{aligned}
            \right.
        \end{equation}
        \item The convection-diffusion equation with Robin boundary conditions
        \begin{equation}
            \label{Conv-DiffEqRobin}
            \left\{
            \begin{aligned}
               &\,\partial_t\rho_t = \sigma\partial_u^2\mf{a}(\rho_t) - p\partial_u \mf{h}(\rho_t), \quad
               \rho_0 = \rho^\mathrm{ini},\\
               &\left[ \sigma\partial_u\mf{a}(\rho_t)-p\mf{h}(\rho_t) \right] \big|_{u=0,1} = 0, \quad \forall\,t\in [0,T].
            \end{aligned}
            \right.
        \end{equation}
        \item The scalar conservation law on a bounded domain
        \begin{equation}
            \label{ScalarConsLawDirichlet}
            \left\{
            \begin{aligned}
               &\,\partial_t\rho_t + p\partial_u \mf{h}(\rho_t) = 0, \quad
               \rho_0 = \rho^\mathrm{ini},\\
               &\,\rho_t(0) = \frac12, \quad \rho_t(1)=1.
            \end{aligned}
            \right.
        \end{equation}
    \end{itemize}
Note that the notions of solutions we consider are weak solutions for \eqref{FastDiffEqNeumann} and \eqref{Conv-DiffEqRobin}, and entropy solutions for \eqref{ScalarConsLawDirichlet}.
The rigorous definition, uniqueness and some basic regularity of these solutions are stated in Section \ref{appendix}.

Let $\mc{M}_1([0,1])$ be the space of measures on $[0,1]$ with total mass bounded by $1$, equipped with the weak topology.
Define $\pi^N:\Omega_N \longrightarrow \mc{M}_1([0,1])$ to be the map that takes a configuration to its empirical measure:
\begin{equation}
        \pi^N(\eta;\diff u):=\frac1N\sum_{x\in\Lambda_N}\eta_x\delta_\frac{x}{N}(\diff u), \qquad \forall\,\eta\in\Omega_N.
\end{equation}
We say that a distribution $\mu^N$ on $\Omega_N$ is associated with a (deterministic) profile $\rho$ if $\pi^N(\eta)$ converges in distribution to $\rho(u)\diff u$ as $N\to\infty$, or equivalently, for any continuous function $G:[0,1]\longrightarrow\R$ and any $\delta>0$,
\begin{equation}
        \label{eq:association}
        \lim_{N\to+\infty}\mu^N \left( \left|\frac1N\sum_{x\in\Lambda_N}G\Big(\frac xN\Big)\eta_x - \int_0^1G(u)\rho(u) \diff u\right| >\delta\right) =0.
\end{equation}

Let $\mu_0^N$ be a probability measure on $\Omega_N$ and $(\eta (t))_{t\ge 0}$ be the Markov process generated by $\mc{L}_N$ in \eqref{def:genFEP} and the initial distribution $\mu_0^N$.
As the map $\phi$ is well-defined only for ergodic configurations, we assume that $\mu_0^N$ is concentrated on the ergodic component, \textit{i.e.}~$\mu_0^N(\mathcal{E}_N)=1$.
We further assume that $\mu_0^N$ is associated with a density profile $\rho^\mathrm{ini}: [0,1]\longrightarrow \big(\frac12 ,1\big]$ in the sense of \eqref{eq:association}. The existence of a measure satisfying both conditions is not straightforward, but it has been proved in \cite[Section 2.3.1]{DCES} that we can build one if we also require the profile $\rho^\mathrm{ini}$ to be continuous.



For $t\ge0$, denote by $\mu_t^N$ the probability distribution of $\eta(t)$.
We are in position to state the main result of this paper, namely the hydrodynamic limit.

\begin{theorem}
    \label{thm:HLFEPs}
    For any continuous function $G:[0,1]\longrightarrow\R$ and any $\delta >0$,
    \begin{equation*}
        \lim_{N\to +\infty} \mu_t^N \left( \left|\frac1N\sum_{x\in\Lambda_N}G\left(\frac xN\right)\eta_x (t) - \int_0^1G(u)\rho_t(u)\diff u\right|>\delta\right) =0
    \end{equation*}
    for all $t\ge 0$, where $\rho_t(u)$ is
    \begin{itemize}
        \item the unique weak solution to \eqref{FastDiffEqNeumann} for \textbf{\emph{(SFEP)}} and \textbf{\emph{(vWAFEP)}},
        \item the unique weak solution to \eqref{Conv-DiffEqRobin} for \textbf{\emph{(WAFEP)}}, 
        \item the unique entropy solution to \eqref{ScalarConsLawDirichlet} for \textbf{\emph{(AFEPvv)}} if furthermore $\kappa\in\big(\frac12,1\big)$.
    \end{itemize}
\end{theorem}

\section{Proof of Theorem \ref{thm:HLFEPs}}
\label{sec:proof}

Since the process $(\eta (t))_{t\ge 0}$ starts from the ergodic component and the total number of particle is conserved by the dynamics, $\xi(t)=\phi(\eta(t))$ is a configuration in $\Omega_M^\ell$, where $(M,\ell)=(M,\ell)(\eta(0))$ are random variables given by
\begin{equation*}
M(\eta):=\sum_x\eta_x+2, \qquad \ell(\eta):=2\sum_x\eta_x-N+2.
\end{equation*}
Combining it with Theorem \ref{thm:micromap}, $(\xi(t))_{t\ge0}$ turns out to be a Markov process evolving on the state space $\phi(\mc{E}_N)$ given in \eqref{eq:state-space} with generator
\begin{equation}
\widetilde{\bb{L}}_Nf(\xi):=\Theta_N \left( \sigma\bb{L}_{M(\phi^{-1}(\xi))}^\mathrm{S}+pN^{-\kappa}\bb{L}_{M(\phi^{-1}(\xi))}^\mathrm{TA} \right),
\end{equation}
for any function $f:\phi(\mc{E}_N)\longrightarrow\bb{R}$.
In other words, $(\xi(t))_{t\ge0}$ is the SEP generated by $\bb{L}_N$ given in \eqref{def:genSEP}, living on the lattice space whose size $M=M(\eta(0))$ is random, yet fixed once the evolution starts.

\medskip

\noindent \textit{Proof of Theorem \ref{thm:HLFEPs}.}
We prove the theorem for all $t\in[0,T]$ where $T>0$ is an arbitrarily fixed time horizon.

    First of all, let $(\eta^N)_{N\ge 1}$ be a \emph{deterministic} sequence of \emph{ergodic} configurations (\textit{i.e.}~$\eta^N\in\mc{E}_N$), such that $\pi^N(\eta^N)$ converges to $\rho^\mathrm{ini}(u)\diff u$ weakly in $\mc{M}_1([0,1])$ as $N\to+\infty$.
    Denote by $(\eta(t))_{t\ge0}$ the FEP generated by $\mc{L}_N$ starting from the configuration $\eta^N$, and by $\P_{\eta^N}$ its distribution on $\mc{D}([0,T],\mc{E}_N)$.
    
    By Theorem \ref{thm:micromap}, $\xi(t)=\phi(\eta(t))$ is a SEP in $\Omega_M$ with initial configuration $\xi^M=\phi(\eta^N)$, where $M=M_N := M(\eta^N)$ is a deterministic sequence of integers.
    In the incoming Proposition \ref{prop:macromap}, we prove that the microscopic mapping $\phi$ induces a macroscopic mapping $\Phi$ between density profiles, and that $\pi^M(\xi^M)$ converges in distribution to $\omega^\mathrm{ini}(v)\diff v$ where $\omega^\mathrm{ini}=\Phi (\rho^\mathrm{ini})$.
    Hence, the hydrodynamics of the SEPs that are given in Proposition \ref{prop:HLSEPs} of Section \ref{sec:HLSEPs} yield that
    \begin{equation*}
        \lim_{N\to +\infty} \mathbb{P}_{\eta^N}\left( \left|\frac1M\sum_{y\in\Lambda_M}G\left(\frac yM\right)\xi_y(t) - \int_0^1G(v)\omega_t(v)\diff v\right|>\delta\right) =0,
    \end{equation*}
    for any $t\ge0$, $\delta>0$ and any continuous function $G$, where $\omega_t$ is the solution of some PDE starting from $\omega^\mathrm{ini}$ defined therein.
    Note that we have used the fact that $M/N$ converges to $m:=\int_0^1 \rho^\mathrm{ini}(u)\diff u$, which is a straightforward consequence of the fact that $\pi^N(\eta^N)$ converges weakly to $\rho^\mathrm{ini}(u)\diff u$.

    Repeating backwards the calculations made to prove Proposition \ref{prop:macromap}, we obtain that
    \begin{equation}
    \label{eq:LLNfromdeterministic}
        \lim_{N\to +\infty} \mathbb{P}_{\eta^N}\left( \left|\frac1N\sum_{x\in\Lambda_N}G\left(\frac xN\right)\eta_x(t) - \int_0^1G(u)\rho_t(u)\diff u\right|>\delta\right) =0
    \end{equation}
    with $\rho_t=\Phi^{-1}(\omega_t)$. In Section \ref{sec:mappingsolutions} we prove that the profile $\rho_t$ defined this way satisfies the PDEs \eqref{FastDiffEqNeumann}, \eqref{Conv-DiffEqRobin} or \eqref{ScalarConsLawDirichlet} depending on the value of the parameters, so this concludes the proof of Theorem \ref{thm:HLFEPs} in the case where we start from a deterministic configuration.
    
    It suffices to extend the result to the initial distribution $\mu_0^N$.
    To this end, observe that the fact that $\mu_0^N$ is associated to $\rho^\mathrm{ini}$ means that, as random variables taking values in $\mc{M}_1([0,1])$,
    \begin{equation*}
        \pi^N(\eta) \xrightarrow[N\to +\infty]{} \rho^\mathrm{ini}(u)\diff u\qquad\mbox{ in }\mu_0^N\mbox{-probability}.
    \end{equation*}
    With Skorokhod's representation theorem (see, \emph{e.g.} \cite[Theorem 1.6.7]{Billingsley}), we can construct a sequence of random elements $(\eta^N)_{N\ge 1}$ on a common probability space $(\Sigma ,\mc{F},\mu^\star)$ such that each $\eta^N$ has law $\mu_0^N$, and for which the convergence above holds $\mu^\star$-almost surely. Therefore, for any $\delta >0$ and any continuous function $G$, we have that 
    \begin{multline*}
        \lim_{N\to +\infty} \P_{\mu_0^N}\left( \left|\frac1N\sum_{x\in\Lambda_N}G\left(\frac xN\right)\eta_x(t)-\int_0^1G(u)\rho_t(u)\diff u\right| >\delta\right) \\
        = \lim_{N\to +\infty}\int_{\Omega_N}\P_{\eta}\left( \left|\frac1N\sum_{x\in\Lambda_N}G\left(\frac xN\right)\eta_x(t)-\int_0^1G(u)\rho_t(u)\diff u\right| >\delta\right)\diff\mu_0^N(\eta )\\
         = \lim_{N\to +\infty} \mathtt{E}_{\mu^\star} \left[\P_{\eta^N}\left( \left|\frac1N\sum_{x\in\Lambda_N}G\left(\frac xN\right)\eta_x(t)-\int_0^1G(u)\rho_t(u)\diff u\right| >\delta\right)\right] 
    \end{multline*}
    where $\mathtt{E}_{\mu^\star}$ denotes the expectation with respect to $\mu^\star$. Using \eqref{eq:LLNfromdeterministic} and the dominated convergence theorem, we get that this last limit is equal to 0, and this concludes the proof of Theorem \ref{thm:HLFEPs}.\hfill $\square$

\subsection{Macroscopic mapping}
\label{sec:macromap}

Recall the map $\phi$ defined in Section \ref{sec:micromap} such that $\phi(\mc{E}_N^k) = \Omega_M^\ell$ with $(M,\ell)=(k+2,2k-N+2)$.
In this section, we show that $\phi$ can be lifted up to a macroscopic map that transfers a density profile to another.

Given $\rho:[0,1]\longrightarrow\big[\frac12,1\big]$, define its total mass and the fraction of it on $[0,u]$:
\begin{equation}
\label{def:v}
        m_\rho := \int_0^1\rho (u)\diff u,\qquad v_\rho (u) := \frac{1}{m_\rho}\int_0^u\rho (u')\diff u'.
\end{equation}
As a continuous, strictly increasing function, $v_\rho$ is invertible and we denote its inverse by $u_\rho=u_\rho(v)$. We then define $\omega=\Phi (\rho):[0,1]\to[0,1]$ by
\begin{equation}
\label{def:macromap}
        \omega (v) := \frac{2\rho (u_\rho (v))-1}{\rho (u_\rho (v))}, \qquad \forall\,v\in[0,1].
\end{equation}

\begin{proposition}
\label{prop:macromap}
Fix $k=k_N$ and let $\mu^N$ be a probability distribution on $\mc{E}_N^k$.
Denote by $\nu^M$ the push-forward distribution of $\mu^N$ under $\phi$.
If $\mu^N$ is associated with a profile $\rho:[0,1]\longrightarrow\big(\frac12,1\big]$ in the sense of \eqref{eq:association},
then so is $\nu^M$ with the profile $\omega=\Phi(\rho)$ given by \eqref{def:macromap}.
\end{proposition}

\begin{remark}
    From \eqref{def:v} and \eqref{def:macromap}, one can see that $u_\rho$ is given by
    \begin{equation}
        \label{def:u}
        u_\rho(v)=m_\rho\int_0^v\big( 2-\omega(v')\big)\diff v'.
    \end{equation}
    It is then straightforward to see that the inverse $\rho=\Phi^{-1}(\omega)$ is given by
    \begin{equation}
    \rho(u)=\frac{1}{2-\omega(v_\rho(u))}, \qquad \forall\,u\in[0,1].
    \end{equation}
\end{remark}

\begin{proof}
Note that $\xi=\phi(\eta)\in\Omega_M^\ell$ where $M=M_N=k_N+2$ and $\ell=\ell_N=2k_N-N+2$ are both deterministic sequences, and as a consequence of the association of $\mu^N$ to $\rho^\mathrm{ini}$, we have that
\begin{equation*}
    \frac MN\xrightarrow[N\to +\infty]{} m_\rho .
\end{equation*}
Since the sequence of measures $(\pi^M(\xi))$ indexed by $N$ is relatively compact in $\mc{M}_1([0,1])$, it admits weak limit points.
As there is at most one particle per site, any weak limit point is a measure which is absolutely continuous with respect to the Lebesgue measure.
Thus, there exists a measurable function $\omega =\omega (v)$ for which, up to an extraction of subsequence,
\begin{equation*}
\lim_{N\to+\infty}\nu^M \left( \left|\frac1M\sum_{y\in\Lambda_M}G\Big(\frac yM\Big)\xi_y - \int_0^1G(v)\omega(v) \diff v\right| >\delta\right) =0
\end{equation*}
for any continuous $G$ and any $\delta>0$. It suffices to show that $\omega$ is given by \eqref{def:macromap}.
Observe that
    \begin{align*}
        \frac1N \sum_{x\in\Lambda_N}G\left(\frac xN\right) \eta_x = \frac1N\sum_{x\in\mathrm{Part}(\eta)}G\left(\frac xN\right) &= \frac1N\sum_{y\in\Lambda_M\setminus\{1\}}G\left(\frac{\gamma_\xi (y)}{N}\right)\\
        &= \frac MN\frac1M\sum_{y\in\Lambda_M\setminus\{1\}}G\left(\frac MN\frac1M \sum_{z=1}^{y-1}(2-\xi_z)\right),
    \end{align*}
    where $\gamma_\xi$ is the bijection defined in \eqref{def:gammaxi}.
    Letting $N$ go to $+\infty$ along the proper subsequence in this equality yields that
    \begin{equation*}
        \int_0^1G(u)\rho(u)\diff u  = m_\rho \int_0^1 G\left( m_\rho\int_0^v \big( 2-\omega(v')\big)\diff v'\right)\diff v.
    \end{equation*}
    By performing the change of variable $u=u_\rho(v)$ defined in \eqref{def:u}, the right-hand side is equal to 
    \begin{equation*}
        \int_0^1 \frac{G(u)}{2-\omega(v_\rho(u))}\diff u.
    \end{equation*}
    As this holds true for any continuous function $G$, the relation between $\omega$ and $\rho$ follows.
    \hfill$\square$
\end{proof}





\subsection{Hydrodynamic limit for SEP}
\label{sec:HLSEPs}

Suppose that $M=M_N$ is a deterministic sequence such that $M/N$ converges to some $m\in [0,1]$ as $N\to +\infty$.
Let $(\xi (t))_{t\ge 0}$ be the SEP starting from some initial distribution $\nu_0^M$ and driven by the generator $\mathbb{L}_N$ defined in \eqref{def:genSEP}. We denote by $\mathbf{P}_{\nu_0^M}$ the probability measure that it induces on the Skorokhod space $\mathcal{D}([0,T],\Omega_M)$.
We collect in the next proposition the hydrodynamic limit for $\xi(t)$ in different parameter regimes.

\begin{proposition}[Hydrodynamic limits for the SEPs]
    \label{prop:HLSEPs}
    Assume that $\nu_0^M$ is associated with a measurable function $\omega^\mathrm{ini}:[0,1]\longrightarrow[0,1]$ in the sense of \eqref{eq:association}.
    Then, for any $\delta >0$ and any continuous function $G$, 
    \begin{equation*}
        \lim_{N\to +\infty} \mathbf{P}_{\nu_0^M}\left( \left|\frac1M\sum_{y\in\Lambda_M}G\left(\frac yM\right)\xi_y (t) - \int_0^1G(v)\omega_t(v)\diff v\right|>\delta\right) =0
    \end{equation*}
    for all $t\in[0,T]$, where $\omega_t=\omega_t(v)$ is given below.
    \begin{itemize}
        \item For \textbf{\emph{(SSEP)}} and \textbf{\emph{(vWASEP)}}, $\omega_t$ is the unique weak solution to the heat equation with Neumann boundary conditions
        \begin{equation}
            \label{HeatEqNeumann}
            \left\{
            \begin{aligned}
               &\,\partial_t\omega_t = \dfrac{\sigma}{m^2}\partial_v^2\omega_t, \quad
               \omega_0=\omega^\mathrm{ini},\\
               &\,\partial_v\omega_t(0)=\partial_v\omega_t(1)=0, \quad \forall\,t\in [0,T].
            \end{aligned}
            \right.
        \end{equation}
        \item For \textbf{\emph{(WASEP)}}, $\omega_t$ is the unique weak solution to the viscous Burgers equation with Robin boundary conditions
        \begin{equation}
            \label{ViscBurgersEqRobin}
            \left\{
            \begin{aligned}
               &\,\partial_t\omega_t = \dfrac{\sigma}{m^2}\partial_v^2\omega_t -\dfrac pm\partial_v J(\omega_t), \quad
               \omega_0=\omega^\mathrm{ini},\\
               &\left[\dfrac{\sigma}{m^2}\partial_v\omega_t-\dfrac pmJ(\omega_t)\right]\Big|_{v=0,1} = 0, \quad \forall\,t\in [0,T],
            \end{aligned}
            \right.
        \end{equation}
        where $J(\omega)=\omega(1-\omega)$ is the macroscopic current.
        \item For \textbf{\emph{(ASEPvv)}} with $\kappa\in(\frac12,1)$, $\omega_t$ is the unique entropy solution to the Burgers equation on bounded domain
        \begin{equation}
            \label{BurgersEqDirichlet}
            \left\{
            \begin{aligned}
               &\,\partial_t\omega_t +\dfrac pm\partial_v J(\omega_t) = 0, \quad
               \omega_0=\omega^\mathrm{ini},\\
               &\,\omega_t(0)=0, \quad \omega_t(1)=1.
            \end{aligned}
            \right.
        \end{equation}
    \end{itemize}
\end{proposition}

\begin{remark}
    The factors $m$ that appear in these equations come from the assumption that $M/N$ converges to $m$, and the fact that the generator $\bb{L}_N$ defined in \eqref{def:genSEP} is accelerated by powers of $N$, whereas its natural scaling parameter is $M$.
\end{remark}

\begin{proof}
    The proof of these results are scattered in different papers.
    We summarize them here.
    
    A remarkable feature of the simple exclusion processes is the \emph{attractiveness}.
    It allows us to apply the result in \cite{Kern}, which states that the closed dynamics of an attractive system is exponentially close to the same system in contact with slow reservoirs, in the sense of large deviations.
    In particular, they share the same hydrodynamic limit.
    Hence, the results for \textbf{(SSEP)}, \textbf{(vWASEP)} and \textbf{(WASEP)} are direct corollaries of those proved for the slow boundaries:
    \begin{itemize}
        \item for \textbf{(SSEP)} it follows from the case $\theta>1$ of \cite[Theorem 2.8]{BMNS17},
        \item for \textbf{(vWASEP)} it follows from the case $\gamma>1$ of \cite[Theorem 1]{CapitaoGoncalves}, and
        \item for \textbf{(WASEP)} it follows from the case $\gamma =1$, $\theta>1$, $\delta >1$ of \cite[Theorem 1]{CapitaoGoncalves}.
    \end{itemize}
    For \textbf{(ASEPvv)}, the result is included in \cite[Theorem 2.8]{Xu_hydrodynamics}. \hfill$\square$
\end{proof}

\section{Mapping of solutions}
\label{sec:mappingsolutions}

We begin with proving that weak solutions are mapped to weak solutions through the mapping $\Phi$. As the proof is more tedious for entropy solutions, we will do it afterwards.

Let $\omega_t(\cdot )$ be defined as in Proposition \ref{prop:HLSEPs}, and let $\rho_t:=\Phi^{-1}(\omega_t)$ for all $t\in [0,T]$, where $\Phi$ is the macroscopic mapping defined through \eqref{def:macromap} in Section \ref{sec:macromap}. We recall that this is a one-to-one mapping and then we have the relations
\begin{equation}
    \label{def:map_rhot_omegat}
    \rho_t(u)=\frac{1}{2-\omega_t(v_t(u))}\qquad\Longleftrightarrow\qquad \omega_t(v)=\frac{2\rho_t(u_t(v))-1}{\rho_t(u_t(v))}=\mf{a}(\rho_t)(u_t(v))
\end{equation}
where
\begin{equation}
    \label{def:ut_vt}
    v_t(u)=\frac1m\int_0^u\rho_t(u')\diff u',\qquad u_t(v)=m\int_0^v\big( 2-\omega_t(v')\big)\diff v'
\end{equation}
and $m$ is the total mass of the initial profile.

\subsection{For weak solutions}

\begin{lemma}[Mapping of weak solutions]
    \label{lemma:weaksol}
    \begin{enumerate}
    \item[(i)] The unique weak solution to the heat equation with Neumann boundary conditions \eqref{HeatEqNeumann} is mapped to the unique weak solution to the fast diffusion equation with Neumann boundary conditions \eqref{FastDiffEqNeumann} through the mapping $\Phi$.
    \item[(ii)] The unique weak solution to the viscous Burgers equations with Robin boundary conditions \eqref{ViscBurgersEqRobin} is mapped to the unique weak solution to the convection-diffusion equation with Robin boundary conditions \eqref{Conv-DiffEqRobin} through the mapping $\Phi$.
    \end{enumerate}
\end{lemma}

\begin{proof}
    We only prove \textit{(ii)} as the proof of \textit{(i)} is the same, but taking $p=0$. Assume that $\omega$ is the weak solution to \eqref{ViscBurgersEqRobin} and define $\rho$ through the formulas \eqref{def:map_rhot_omegat} and \eqref{def:ut_vt}. On the one hand, we have that
    \begin{align}
        \partial_v\omega_t(v)=\partial_v \big(\mf{a}\circ\rho_t\circ u_t (v)\big) & = \partial_u\mf{a}(\rho_t)(u_t(v))\times\partial_vu_t(v) \notag \\
        & = \partial_u\mf{a}(\rho_t)(u_t(v))\times m (2-\omega_t(v)) \notag\\
        & = m\frac{\partial_u\mf{a}(\rho_t)}{\rho_t}(u_t(v)) \label{eq:dvomega}
    \end{align}
    and then
    \begin{align}
        \partial_v^2\omega_t (v) & = m\partial_vu_t(v)\times \partial_u\left(\frac{\partial_u\mf{a}(\rho_t)}{\rho_t}\right) (u_t(v)) \notag \\
        & = m^2\left( \frac{\partial_u^2\mf{a}(\rho_t)}{\rho_t^2}-\frac{\partial_u\mf{a}(\rho_t)\partial_u\rho_t}{\rho_t^3}\right) (u_t(v)). \label{eq:dv2omega}
    \end{align}
    On the other hand, we have
    \begin{align*}
        \partial_t\omega_t (v) = \partial_t\big( \mf{a}\circ\rho_t\circ u_t(v)\big) & = \mf{a}'(\rho_t(u_t(v)))\times\partial_t(\rho_t(u_t(v)))\\
        & = \frac{1}{\rho_t(u_t(v))^2} \big( \partial_t\rho_t(u_t(v))+\partial_u\rho_t(u_t(v))\partial_tu_t(v)\big),
    \end{align*}
    but using the equation \eqref{ViscBurgersEqRobin},
    \begin{align*}
        \partial_tu_t(v) = -m\int_0^v \partial_t\omega_t(v')\diff v' & = -m\int_0^v\left( \frac{\sigma}{m^2}\partial_v^2\omega_t(v')-\frac pm\partial_v J(\omega_t)(v')\right)\diff v'\\
        & = pJ(\omega_t)(v)-\frac{\sigma}{m}\partial_v\omega_t(v)
    \end{align*}
    An easy computation shows that 
    \begin{equation}
        \label{eq:dvJ(omega)}
        J(\omega_t)(v)=\frac{\mf{h}(\rho_t)}{\rho_t}(u_t(v))\quad\mbox{ so }\quad \partial_vJ(\omega_t)(v)=\frac{m}{\rho_t}\partial_u\left( \frac{\mf{h}(\rho_t)}{\rho_t}\right) (u_t(v)).
    \end{equation}
    As a consequence, using \eqref{eq:dvomega},
    \begin{equation*}
        \partial_tu_t(v) = \frac{p\mf{h}(\rho_t)-\sigma\mf{a}(\rho_t)}{\rho_t}(u_t(v))
    \end{equation*}
    and hence
    \begin{equation}
        \label{eq:dtomega}
        \partial_t\omega_t(v)= \left( \frac{\partial_t\rho_t}{\rho_t^2}+ \frac{p\mf{h}(\rho_t)-\sigma\partial_u\mf{a}(\rho_t)}{\rho_t^3}\partial_u\rho_t\right) (u_t(v)).
    \end{equation}
    If we put \eqref{eq:dv2omega}, \eqref{eq:dvJ(omega)} and \eqref{eq:dtomega} together in equation \eqref{ViscBurgersEqRobin}, we get that 
    \begin{equation*}
        \frac{\partial_t\rho_t}{\rho_t^2} + \frac{p\mf{h}(\rho_t)-\sigma\partial_u\mf{a}(\rho_t)}{\rho_t^3}\partial_u\rho_t = \sigma\left( \frac{\partial_u^2\mf{a}(\rho_t)}{\rho_t^2}-\frac{\partial_u\mf{a}(\rho_t)\partial_u\rho_t}{\rho_t^3}\right) -\frac{p}{\rho_t}\partial_u\left(\frac{\mf{h}(\rho_t)}{\rho_t}\right)
    \end{equation*}
    and this rewrites exactly 
    \begin{align*}
        \partial_t\rho_t & = \sigma \partial_u^2\mf{a}(\rho_t) - p\left( \rho_t\partial_u\left(\frac{\mf{h}(\rho_t)}{\rho_t}\right) +\frac{\mf{h}(\rho_t)}{\rho_t}\partial_u\rho_t\right) \\
        & = \sigma\partial_u^2\mf{a}(\rho_t) - p\partial_u\mf{h}(\rho_t)
    \end{align*}
    as expected. By \eqref{eq:dvomega} and \eqref{eq:dvJ(omega)}, the Robin boundary conditions rewrite exactly as expected also. \hfill$\square$.
\end{proof}

\subsection{For entropy solutions}

\begin{lemma}[Mapping of entropy solutions]
\label{lemma:entropysol}
    The unique entropy solution $\omega$ to the Burgers equation with Dirichlet boundary conditions \eqref{BurgersEqDirichlet} is mapped to the unique entropy solution $\rho$ to the scalar conservation law with Dirichlet boundary conditions \eqref{ScalarConsLawDirichlet} through the mapping $\Phi$.
\end{lemma}

\begin{proof}
    Let $\omega$ be the unique entropy solution $\omega$ to the Burgers equation with Dirichlet boundary conditions \eqref{BurgersEqDirichlet}. By Theorem \ref{thm:existence_entropysol} in Section \ref{appendix}, we know that its parabolic perturbation
    \begin{equation}
        \label{eq:ParabolicPertBurgers}
        \left\{\begin{array}{l}
           \partial_t\omega^\varepsilon +\dfrac pm\partial_v J(\omega^\varepsilon ) =\dfrac{\varepsilon}{m^2}\partial_v^2\omega^\varepsilon, \\
           \omega_0^\varepsilon(\cdot )=\omega^\mathrm{ini}(\cdot ),\\
           \omega_t^\varepsilon(0)=0\mbox{ and }\omega_t^\varepsilon(1)=1\qquad\mbox{ for all }t\in [0,T]
        \end{array}\right.
    \end{equation}
    admits a unique smooth solution and that, up to an extraction,
    \begin{equation*}
        \omega^\varepsilon\xrightarrow[\varepsilon\to 0]{}\omega\qquad\mbox{ in } L^1([0,T]\times [0,1])\mbox{ and almost-everywhere in }[0,T]\times [0,1].
    \end{equation*}
    Let $\rho^\varepsilon$ and $\rho$ be the inverse images of $\omega^\varepsilon$ and $\omega$ respectively by the mapping $\Phi$, \textit{i.e.}~$\rho_t^\varepsilon = \Phi^{-1}(\omega_t^\varepsilon)$ and $\rho_t=\Phi^{-1}(\omega_t)$ for all $t\in [0,T]$. First, let us prove that $\rho^\varepsilon$ converges to $\rho$ in $L^1([0,T]\times [0,1])$ as $\varepsilon\to 0$, and for this it is sufficient to prove the convergence of $\omega^\varepsilon\circ v^\varepsilon$ towards $\omega\circ v$ where $v^\varepsilon$ is defined by 
    \begin{equation*}
        v_t^\varepsilon (u)=\frac1m\int_0^u\rho_t^\varepsilon (u')\diff u'.
    \end{equation*}
    As usual, the inverse function $u_t^\varepsilon (v)$ of $v_t^\varepsilon (u)$ is given by
    \begin{equation*}
        u_t^\varepsilon (v)=m\int_0^v\big( 2-\omega_t^\varepsilon (v')\big)\diff v'
    \end{equation*}
    so for any $t\in [0,T]$, and any $v\in [0,1]$,  we have 
    \begin{equation*}
        |u_t^\varepsilon (v)-u_t(v)|\le m\int_0^v|\omega_t^\varepsilon (v')-\omega_t (v')|\diff v'\le m \|\omega_t^\varepsilon -\omega_t\|_{L^1([0,1])}.
    \end{equation*}
    As $\omega^\varepsilon$ converges to $\omega$ in $L^1$, the right-hand side of this inequality converges to 0 and it proves that $u_t^\varepsilon$ converges uniformly to $u_t$ on $[0,1]$ for any $t\in [0,T]$. Since $u_t^\varepsilon$ and $u_t$ are strictly increasing functions, we get that the inverse $v_t^\varepsilon$ also converges uniformly to $v_t$ on $[0,1]$ for any $t\in [0,T]$. Therefore,
    \begin{align*}
        \| \omega_t^\varepsilon \circ v_t^\varepsilon - \omega_t\circ v_t\|_{L^1([0,1])} & \le \|\omega_t^\varepsilon\circ v_t^\varepsilon -\omega_t^\varepsilon\circ v_t\|_{L^1([0,1])}+\|\omega_t^\varepsilon\circ v_t-\omega_t\circ v_t\|_{L^1([0,1])}\\
        & \le \|\partial_v\omega_t^\varepsilon\|_{L^\infty ([0,1])}\| v_t^\varepsilon - v_t\|_{L^\infty([0,1])} + \frac1m\|\omega_t^\varepsilon -\omega_t\|_{L^1([0,1])}.
    \end{align*}
    where to obtain last inequality we used the mean-value theorem for the first term, and a change a variable with the fact that $\rho$ takes value in $[0,1]$ for the second term. The first term on the right-hand side of this inequality vanishes as $\varepsilon\to 0$ as $\omega^\varepsilon$ is smooth and $v_t^\varepsilon$ converges uniformly to $v_t$. The remaining term also vanishes as $\omega^\varepsilon$ converges to $\omega$ in $L^1$, so it proves that 
    \begin{equation*}
        \omega_t^\varepsilon\circ v_t^\varepsilon\xrightarrow[\varepsilon\to 0]{} \omega_t\circ v_t\qquad\mbox{ in }L^1([0,1])
    \end{equation*}
    for any $t\in [0,T]$. As $\omega^\varepsilon$ is bounded by 1, we can use the dominated convergence theorem to deduce that $\omega^\varepsilon\circ v^\varepsilon$ converges to $\omega\circ v$ in $L^1([0,T]\times [0,1])$, and then deduce that $\rho^\varepsilon$ converges to $\rho$ in $L^1$. We are now ready to prove that $\rho$ is the unique entropy solution to \eqref{ScalarConsLawDirichlet}, beginning with the entropy inequality.

    \medskip

    Using \eqref{eq:ParabolicPertBurgers} and similar computations to the one made in the proof of Lemma \ref{lemma:weaksol}, we can show that $\rho^\varepsilon$ is a weak solution to 
    \begin{equation}
        \label{eq:ParabolicPertConv-Diff}
        \left\{\begin{array}{l}
           \partial_t\rho^\varepsilon +p\partial_u \mf{h}(\rho^\varepsilon ) =\varepsilon\partial_u^2\mf{a}(\rho^\varepsilon), \\
           \rho_0^\varepsilon(\cdot )=\rho^\mathrm{ini}(\cdot ),\\
           \rho_t^\varepsilon(0)=\frac12\mbox{ and }\rho_t^\varepsilon(1)=1\qquad\mbox{ for all }t\in [0,T].
        \end{array}\right.
    \end{equation}
    Let $(F,Q)$ be a Lax entropy-flux pair associated to $\mf{h}$. If we multiply both sides of the PDE \eqref{eq:ParabolicPertConv-Diff} by $F'(\rho_t^\varepsilon )$, we can rewrite it as
    \begin{align*}
        \partial_tF(\rho_t^\varepsilon )+p\partial_u Q(\rho_t^\varepsilon ) = \varepsilon F'(\rho_t^\varepsilon )\partial_u^2\left(\frac{2\rho_t^\varepsilon -1}{\rho_t^\varepsilon}\right) & = \varepsilon\partial_u^2\tilde{F}(\rho_t^\varepsilon) - \varepsilon \frac{F''(\rho_t^\varepsilon)}{(\rho_t^\varepsilon )^2}(\partial_u\rho_t^\varepsilon )^2\\
        & \le \varepsilon \partial_u^2\tilde{F}(\rho_t^\varepsilon )
    \end{align*}
    where $\tilde{F}(r)$ is a primitive of $\frac{F'(r)}{r^2}$, and in the last inequality we used the fact that $F$ is convex. If we multiply both sides of this inequality by a non-negative, compactly supported smooth test function $\varphi\in \mc{C}_c^\infty ([0,T]\times [0,1])$ and we integrate by parts, this inequality rewrites
    \begin{equation}
        \label{eq:step1_entropyineq}
        \int_0^T \Big( \big\langle \partial_t\varphi_s,F(\rho_s^\varepsilon)\big\rangle + p\big\langle \partial_u\varphi_s,Q(\rho_s^\varepsilon)\big\rangle\Big)\diff s \ge -\varepsilon \int_0^T \big\langle \partial_u^2\varphi_s ,\tilde{F}(\rho_s^\varepsilon )\big\rangle\diff s
    \end{equation}
    in the sense of distribution, where $\langle \cdot ,\cdot\rangle$ denotes the standard scalar product in $L^2([0,1])$. The right-hand side of this inequality converges clearly to 0, and as $\rho^\varepsilon$ converges to $\rho$ in $L^1$, it is also straightforward to check that 
    \begin{align*}
        \left\{\begin{array}{ll}
            F(\rho^\varepsilon )\xrightarrow[\varepsilon\to 0]{} F(\rho ), & \\
            & \quad\mbox{ in }L^1([0,T]\times [0,1]).\\
            Q(\rho^\varepsilon )\xrightarrow[\varepsilon\to 0]{} Q(\rho )&
        \end{array}\right.
    \end{align*}
    Letting $\varepsilon$ go to 0 in \eqref{eq:step1_entropyineq} proves that $\rho$ satisfies the entropy inequality (i) of Definition~\ref{defin:entropysol}. It also satisfies the $L^1$-initial condition (ii) as $\rho^\varepsilon$ converges to $\rho$ in $L^1$.

    \medskip

    It remains to prove that $\rho$ satisfies Otto-type boundary conditions (iii). As $\omega$ satisfies the entropy inequality, it is known (see \cite[Proposition 2.5]{Xu_hydrodynamics}) that it admits boundary traces, \textit{i.e.}~there are functions $\omega_\cdot (\cdot )$ and $\omega_\cdot (1)\in L^\infty ([0,T])$ such that 
    \begin{equation*}
        \lim_{v\to 0^+} \int_0^T |\omega_t(v)-\omega_t(0)|\diff t=0\quad\mbox{and}\quad \lim_{v\to 1^-}\int_0^T |\omega_t(v)-\omega_t(1)|\diff t =0.
    \end{equation*}
    For the same reason, $\rho$ admits boundary traces $\rho_\cdot (0)$ and $\rho_\cdot (1)$, and using the mapping between $\rho$ and $\omega$, we easily get the relations
    \begin{equation}
        \label{eq:relation_boundarytraces}
        \rho_t(0)=\frac{1}{2-\omega_t(0)}\qquad\mbox{ and }\qquad \rho_t(1)=\frac{1}{2-\omega_t(1)}\quad\mbox{ for all }t\in [0,T].
    \end{equation}
    Let $F$ be a $\mc{C}^2$ function which satisfies
    \begin{equation*}
        \forall r\le 0,\; F(r)=0,\qquad \forall r>1,\; F(r)=r\quad\mbox{ and }\quad \forall r\in\R,\; F''(r)\ge 0.
    \end{equation*}
    Then, for any $q\in [0,1]$, any $\gamma >0$, define
    \begin{equation*}
        \forall r\in [0,1],\qquad F_{\gamma ,q}(r)= \frac{1}{\gamma} F\big( \gamma (r-q)\big),
    \end{equation*}
    and also 
    \begin{equation*}
        \forall r\in [0,1],\qquad Q_{\gamma ,q}(r)= \int_q^r J'(\tilde{r}) F'\big( \gamma (\tilde{r}-q)\big)\diff\tilde{r}.
    \end{equation*}
    The couple $(F_{\gamma,q},Q_{\gamma ,q})$ is a Lax entropy-flux pair associated to $J$, and it is not difficult to check that the couple $(\mc{F}_\gamma ,\mc{Q}_\gamma)(r,q)= (F_{\gamma ,q},Q_{\gamma ,q})(r)$ defines a boundary-entropy flux pair. In particular, since $\omega$ satisfies the Otto-type boundary conditions, we have that 
    \begin{equation*}
        \lim_{v\to 0^+}\int_0^T\varphi (t)\mc{Q}_\gamma\big( \omega_t(v),0\big)\diff t = \int_0^T \varphi (t)\mc{Q}_\gamma \big( \omega_t (0),0\big)\diff t\le 0
    \end{equation*}
    for any continuous and compactly supported function ${\varphi : [0,T]\longrightarrow\R_+}$. As a consequence,
    \begin{equation}
        \label{eq:step1_Ottobound}
        \mc{Q}_\gamma\big( \omega_t(0),0\big) \le 0 \quad\mbox{ for almost every }t\in [0,T].
    \end{equation}
    Letting $\gamma$ go to infinity in the definition of $F_{\gamma ,q}$ and $Q_{\gamma ,q}$, one can see that 
    \begin{equation*}
        \lim_{\gamma\to +\infty}F_{\gamma ,q}(r)=(r-q)\mathds{1}_{r\ge q}\qquad\mbox{ and }\qquad\lim_{\gamma\to +\infty}Q_{\gamma ,q}(r) = \big( J(r)-J(q)\big)\mathds{1}_{r\ge q}. 
    \end{equation*}
    Therefore, letting $\gamma$ go to infinity in inequality \eqref{eq:step1_Ottobound} yields that $J\big( \omega_t(0)\big)\mathds{1}_{\omega_t(0)\ge 0}\le 0$ for almost every $t\in [0,T]$, which implies, as $\omega_t(0)$ takes values in $[0,1]$ that $\omega_t(0)$ is either 0 or 1 for almost every $t\in [0,T]$. In a similar way, we can prove that $\omega_t(1)\in\{ 0,1\}$ for almost every $t\in [0,T]$. Thanks to relation \eqref{eq:relation_boundarytraces}, we deduce that $\rho_t(0)$ and $\rho_t(1)$ have value either $\frac12$ or $1$ for almost every $t\in [0,T]$.

    We are now in position to prove that $\rho$ satisfies the Otto-type boundary conditions, we prove it only for the left boundary, but the proof on the right is exactly the same. Let $\mc{F},\mc{Q}:[0,1]^2\longrightarrow\R$ be a boundary entropy-flux pair associated to $\mf{h}$. By properties, we have that
    \begin{equation*}
        \mc{Q}\big(\rho_t(0),\tfrac12\big) = \int_\frac12^{\rho_t(0)} \partial_r\mc{Q}\big( r,\tfrac12\big)\diff r = \int_\frac12^{\rho_t(0)} \mf{h}'(r)\partial_r\mc{F}\big( r,\tfrac12\big)\diff r.
    \end{equation*}
    When $\rho_t(0)=\frac 12$ this integral is equal to 0 and then \eqref{eq:OttoBC} is satisfied, so we only have to treat the case when $\rho_t(0)=1$. In this case, an integration by parts in the last integral gives that
    \begin{align*}
        \mc{Q}\big( \rho_t(0),\tfrac12\big) & = \big[ \mf{h}(r)\partial_r\mc{F}\big(r,\tfrac12\big)\big]_\frac12^1-\int_\frac12^1 \mf{h}(r)\partial_r^2\mc{F}\big( r,\tfrac12\big)\diff r\\
         & =-\int_\frac12^1 \mf{h}(r)\partial_r^2\mc{F}\big( r,\tfrac12\big)\diff r\\
         &\le 0
    \end{align*}
    as $\mf{h}$ vanishes at $\tfrac12$ and $1$, and the function $\mc{F}\big( \cdot ,\tfrac12\big)$ is convex. This proves the Otto-type boundary condition \eqref{eq:OttoBC} on the left boundary.\hfill $\square$
\end{proof}

\section{Notions of solution}
\label{appendix}

\subsection{Weak solutions}

We define here the notion of weak solutions that we use in Section \ref{sec:hydrodynamic} for the generic quasi-linear convection-diffusion equation
\begin{equation}
    \label{genericConvDiffEq}
    \left\{\begin{array}{l}
        \partial_t\rho = \sigma\partial_u^2 \mf{a}(\rho )-p\partial_u\mf{h}(\rho )\qquad\mbox{ on }[0,T]\times [0,1],\\
        \rho_0(\cdot )=\rho^\mathrm{ini}(\cdot )
    \end{array}\right.
\end{equation}
where $\mf{a}$ and $\mf{h}$ are general smooth functions, $\sigma$ and $p$ are non-negative real numbers, and $\rho^\mathrm{ini}:[0,1]\longrightarrow [0,1]$ is a measurable initial profile.

The space of test functions that we consider is the space $\mc{C}^{1,2}([0,T]\times [0,1])$ composed of functions $G:(t,u)\in[0,T]\times [0,1]\longmapsto G_t(u)\in\R$ that are of class $\mc{C}^1$ with respect to the time variable, and of class $\mc{C}^2$ with respect to the space variable.

We first consider the case of Robin boundary conditions, that is necessary for defining solutions of equations \eqref{Conv-DiffEqRobin} and \eqref{ViscBurgersEqRobin} in the weakly asymmetric regime.

\begin{definition}[Weak solution of \eqref{genericConvDiffEq} with Robin boundary conditions]
    \label{defin:weaksolRobin}
    We say that a function $\rho : [0,T]\times [0,1]\longrightarrow\R$ is a \emph{weak solution} to the convection-diffusion equation \eqref{genericConvDiffEq} with \emph{Robin boundary conditions}
    \begin{equation*}
        \Big( \sigma\partial_u\mf{a}(\rho_t)(u)=p\mf{h}(\rho_t)(u)\Big)\Big|_{u=0,1}\qquad\mbox{ for all }t\in [0,T]
    \end{equation*}
    if for any test function $G\in\mc{C}^{1,2}([0,T]\times [0,1])$ and any $t\in [0,T]$, we have
    \begin{multline}
        \label{eq:weaksolRobin}
        \int_0^1\rho_t(u)G_t(u)\diff u - \int_0^1\rho^\mathrm{ini}(u)G_0(u)\diff u - \int_0^t\int_0^1 \rho_s(u)\partial_tG_s(u)\diff u\diff s\\
        =\sigma\int_0^t\int_0^1\mf{a}(\rho_s(u))\partial_u^2G_s(u)\diff u \diff s + p\int_0^t\int_0^1\mf{h}(\rho_s(u))\partial_uG_s(u)\diff u\diff s
        \\- \sigma\int_0^t \big\{ \mf{a}(\rho_s(1))\partial_uG_s(1)-\mf{a}(\rho_s(0))\partial_uG_s(0)\big\}\diff s.
    \end{multline}
\end{definition}

For the symmetric and very weakly asymmetric regimes, we need to introduce the notion of weak solutions to the non-linear diffusion equation
\begin{equation}
    \label{NLDiffEq}
    \left\{\begin{array}{l}
        \partial_t\rho = \sigma\partial_u^2 \mf{a}(\rho )\qquad\mbox{ on }[0,T]\times [0,1],\\
        \rho_0(\cdot )=\rho^\mathrm{ini}(\cdot )
    \end{array}\right.
\end{equation}
with Neumann boundary conditions. This allows to define solutions of \eqref{HeatEqNeumann} and \eqref{FastDiffEqNeumann}.

\begin{definition}[Weak solution of \eqref{NLDiffEq} with Neumann boundary conditions]
    \label{defin:weaksolNeumann}
    We say that a function $\rho : [0,T]\times [0,1]\longrightarrow\R$ is a \emph{weak solution} to the non-linear diffusion equation \eqref{NLDiffEq} with \emph{Neumann boundary conditions}
    \begin{equation*}
        \partial_u\mf{a}(\rho_t)(0)=\partial_u\mf{a}(\rho_t)(1)=0\qquad\mbox{ for all }t\in [0,T]
    \end{equation*}
    if for any test function $G\in\mc{C}^{1,2}([0,T]\times [0,1])$ and any $t\in [0,T]$, equation \eqref{eq:weaksolRobin} with $p=0$ is satisfied.
\end{definition}

\subsection{Entropy solutions}

We now turn to the more elaborate notion of \emph{entropy solutions} that will be useful to define solutions of equations \eqref{BurgersEqDirichlet} and \eqref{ScalarConsLawDirichlet}. Consider the generic scalar conservation law in the bounded domain $[0,1]$
\begin{equation}
    \label{genericConsLawDir}
    \left\{\begin{array}{ll}
        \partial_t\rho +p\partial_u\mf{h}(\rho )=0 & \mbox{ on }[0,T]\times [0,1],\\
        \rho_0(\cdot )=\rho^\mathrm{ini}(\cdot ),& \\
        \rho_t(0)=\rho_-\mbox{ and }\rho_t(1)=\rho_+ & \mbox{ for all }t\in [0,T]
    \end{array}\right.
\end{equation}
where $\mf{h}$ is a smooth function, $\rho_-,\rho_+\in [0,1]$ and $\rho^\mathrm{ini}$ is a fixed initial profile. We follow \cite[Section 2.6]{MalekPDE} to introduce the notion of entropy solutions.

\begin{definition}[Lax Entropy-Flux pair]
    \label{defin:LaxEntropy-Flux}
    A pair of functions $F,Q:[0,1]\longrightarrow\R$ is called a \emph{Lax Entropy-Flux pair} associated to $\mf{h}$ if:
    \begin{itemize}
        \item They are of class $\mc{C}^2$,
        \item $F''(r)\ge 0$ for all $r\in [0,1]$, \textit{i.e.}~$F$ is convex,
        \item $Q'(r)=\mf{h}'(r)F'(r)$ for all $r\in [0,1]$.
    \end{itemize}
\end{definition}

\begin{definition}[Boundary Entropy-Flux pair]
    \label{defin:BoundaryEntropy-Flux}
    A pair $\mc{F},\mc{Q}:[0,1]^2\longrightarrow\R$ is called a \emph{boundary entropy-flux pair} if $(\mc{F},\mc{Q})(\cdot ,q)$ is a Lax entropy-flux pair for all ${q\in [0,1]}$, and if
    \begin{equation}
        \forall q\in [0,1],\qquad \mc{F}(q,q)=\partial_r\mc{F}(r,q)|_{r=q}=\mc{Q}(q,q)=0.
    \end{equation}
\end{definition}

\begin{definition}[Entropy solution to \eqref{genericConsLawDir} with Dirichlet boundary conditions] 
    \label{defin:entropysol}
    A function $\rho :[0,T]\times [0,1]\longrightarrow [0,1]$ is called an \emph{entropy solution} to \eqref{genericConsLawDir} if the following conditions are satisfied:
    \begin{enumerate}
        \item[(i)] \emph{Entropy inequality}: for all Lax entropy-flux pair $(F,Q)$ associated to $\mf{h}$,
        \begin{equation}
            \label{eq:entropyinequality}
            \partial_tF(\rho )+p\partial_u Q(\rho )\le 0
        \end{equation}
        as a distribution on $(0,T)\times (0,1)$;
        \item[(ii)] \emph{$L^1$-initial condition}:
        \begin{equation}
            \label{eq:L1-IC}
            \lim_{t\to 0^+}\int_0^1|\rho_t(u)-\rho^\mathrm{ini}(u)|\diff u =0;
        \end{equation}
        \item[(iii)] \emph{Otto-type boundary conditions}: for all boundary flux $\mc{Q}$ associated to $\mf{h}$ and for all non-negative, continuous and compactly supported function $\varphi : \R_+\longrightarrow\R_+$,
        \begin{equation}
            \label{eq:OttoBC}
            \begin{array}{c}
                \displaystyle\lim_{x\to 0^+}\int_0^T \varphi (t)\mc{Q}\big( \rho_t(x),\rho_-\big)\diff t \le 0, \\
                \\
                \displaystyle\lim_{x\to 1^-}\int_0^T \varphi (t)\mc{Q}\big( \rho_t(x),\rho_+\big)\diff t \ge 0.
            \end{array}
        \end{equation}
    \end{enumerate}
\end{definition}

Take $\varepsilon >0$ and consider the following parabolic perturbation of \eqref{genericConsLawDir}
\begin{equation}
    \label{genericParabolicPert}
    \left\{\begin{array}{ll}
        \partial_t\rho^\varepsilon +p\partial_u\mf{h}(\rho^\varepsilon )=\varepsilon\partial_u^2\rho^\varepsilon & \mbox{ on }[0,T]\times [0,1],\\
        \rho_0^\varepsilon (\cdot )=\rho^\mathrm{ini}(\cdot ),& \\
        \rho_t^\varepsilon (0)=\rho_-\mbox{ and }\rho_t^\varepsilon(1)=\rho_+ & \mbox{ for all }t\in [0,T]
    \end{array}\right.
\end{equation}
It is known that \eqref{genericParabolicPert} admits a unique smooth solution $\rho^\varepsilon$ (for the proof, one can get inspired of the proof of \cite[Lemma 2.2.3]{MalekPDE}). Then, we have the following result which is an adaptation of \cite[Theorem 2.8.20]{MalekPDE}:

\begin{theorem}[Existence of an entropy solution]
    \label{thm:existence_entropysol}
    Under the assumption that $\mf{h}$ is of class $\mc{C}^2$ and that the function $\rho^\mathrm{ini}$ is bounded, the family of smooth solutions $(\rho^\varepsilon)_{\varepsilon >0}$ to the parabolic perturbation \eqref{genericParabolicPert} is uniformly bounded and converges in $\mc{C}^0([0,T],L^1([0,1]))$ to some function $\rho\in L^\infty ([0,T]\times [0,1])$ which is an entropy solution to the problem \eqref{genericConsLawDir} in the sense of Definition \ref{defin:entropysol}. In particular, there exists an entropy solution to \eqref{genericConsLawDir}.
\end{theorem}
Then, we state the following result which is a corollary of \cite[Theorem 2.7.28]{MalekPDE}:
\begin{theorem}
    If $\mf{h}$ is of class $\mc{C}^1$, and if the initial condition $\rho^\mathrm{ini}$ is bounded, there is at most one entropy solution to the problem \eqref{genericConsLawDir}.
\end{theorem}

\end{document}